\documentclass[a4paper,12pt]{article}
\usepackage{graphicx}
\usepackage{amsfonts}
\usepackage{amsthm}
\usepackage{pgfplots}
\usepackage{amsmath}
\usepackage{amssymb}
\usepackage{tikz}
\usetikzlibrary{positioning,decorations.pathmorphing}

\tikzset{main node/.style={circle,fill=black,draw,minimum width=4pt,inner sep=0pt}}

\newenvironment{customlegend}[1][]{\begingroup \csname pgfplots@init@cleared@structures\endcsname \pgfplotsset{#1}}{\csname pgfplots@createlegend\endcsname \endgroup}
\def\addlegendimage{\csname pgfplots@addlegendimage\endcsname}

\title{\textbf{Maximising $H$-Colourings of Graphs}}
\author{Hannah Guggiari\thanks{Mathematical Institute, University of Oxford, Oxford, OX2 6GG, United Kingdom.}, Alex Scott\footnotemark[1] \thanks{Supported by a Leverhulme Trust Research Fellowship.}}
\date{\today}

\makeatletter
\def\thmhead@plain#1#2#3{\thmname{#1}\thmnumber{\@ifnotempty{#1}{ }\@upn{#2}}\thmnote{ {\the\thm@notefont#3}}}
\let\thmhead\thmhead@plain
\makeatother

\theoremstyle{theorem}
\newtheorem{theorem}{Theorem}[section]
\newtheorem{lemma}[theorem]{Lemma}
\newtheorem{conjecture}[theorem]{Conjecture}
\newtheorem{corollary}[theorem]{Corollary}
\newtheorem{proposition}[theorem]{Proposition}

\theoremstyle{definition}
\newtheorem*{definition}{Definition}
\newtheorem*{convention}{Convention}

\begin{document}
\newpage
\maketitle

\begin{abstract}
\noindent
For graphs $G$ and $H$, an $H$-colouring of $G$ is a map $\psi:V(G)\rightarrow V(H)$ such that $ij\in E(G)\Rightarrow\psi(i)\psi(j)\in E(H)$. The number of $H$-colourings of $G$ is denoted by $\hom(G,H)$.

We prove the following: for all graphs $H$ and $\delta\geq3$, there is a constant $\kappa(\delta,H)$ such that, if $n\geq\kappa(\delta,H)$, the graph $K_{\delta,n-\delta}$ maximises the number of $H$-colourings among all connected graphs with $n$ vertices and minimum degree $\delta$. This answers a question of Engbers.

We also disprove a conjecture of Engbers on the graph $G$ that maximises the number of $H$-colourings when the assumption of the connectivity of $G$ is dropped.

Finally, let $H$ be a graph with maximum degree $k$. We show that, if $H$ does not contain the complete looped graph on $k$ vertices or $K_{k,k}$ as a component and $\delta\geq\delta_0(H)$, then the following holds: for $n$ sufficiently large, the graph $K_{\delta,n-\delta}$ maximises the number of $H$-colourings among all graphs on $n$ vertices with minimum degree $\delta$. This partially answers another question of Engbers.
\end{abstract}

\section{Introduction}
Let $G$ be a simple, loopless graph and let $H$ be a simple graph, possibly with loops. A \textit{graph homomorphism} from $G$ to $H$ is a map $\psi:V(G)\rightarrow V(H)$ such that $ij\in E(G)\Rightarrow\psi(i)\psi(j)\in E(H)$. An \textit{$H$-colouring} of $G$ is a graph homomorphism from $G$ to $H$. We denote by $\hom(G,H)$ the number of $H$-colourings of $G$.

Given a class of graphs $\mathcal{G}$ and a fixed graph $H$, it is natural to ask which $G\in\mathcal{G}$ maximises $\hom(G,H)$. Various classes of graphs have been considered (see Cutler \cite{C12} for a survey). For instance, a number of authors, such as Galvin \cite{G13}, have studied the class of all $\delta$-regular graphs for fixed $\delta$; others, including Loh, Pikhurko and Sudakov \cite{LPS10}, have investigated the class of all graphs with $n$ vertices and $m$ edges. In this paper, we consider the class of all graphs with minimum degree at least $\delta$. This class was studied by Engbers \cite{E15,E17} who raised a number of questions and conjectures. We will answer two of these and provide a partial answer to a third.

In Section \ref{sec:connectedgraph}, we consider the case when $\mathcal{G}$ is the set of all \textit{connected} graphs on $n$ vertices with minimum degree at least $\delta$. For this $\mathcal{G}$ and any non-regular graph $H$, Engbers \cite{E17} showed that, for any fixed $\delta\geq2$ and $n$ sufficiently large, $\hom(G,H)$ is maximised uniquely by $G=K_{\delta,n-\delta}$. In this paper, we will extend this result by showing that it holds for all $\delta\geq3$ and for all graphs $H$. This answers a question posed by Engbers \cite{E17}. In the case where $\delta=2$ and $H$ is any graph, Engbers \cite{E15} showed that the number of $H$-colourings is maximised by one of $K_{2,n-2}$, $\frac n3K_3$ or $\frac n4K_{2,2}$ (depending on the structure of $H$).

An $H$-colouring of $G$ requires that each component of $G$ is mapped to a component of $H$. As we are only considering connected graphs $G$, each $H$-colouring of $G$ maps $G$ to a single component of $H$. We therefore begin with the case when $H$ is connected.

\begin{theorem}
\label{thm:connectedgraph}
For every $\delta\geq3$ and every connected graph $H$, there exists a constant $\kappa(\delta,H)$ such that the following holds: if $n\geq\kappa(\delta,H)$ and $G$ is a connected graph on $n$ vertices with minimum degree at least $\delta$, then we have $\hom(G,H)\leq\hom(K_{\delta,n-\delta},H)$. Further, if $H$ is not a complete looped graph or a complete balanced bipartite graph, we have equality if and only if $G=K_{\delta,n-\delta}$.
\end{theorem}
\noindent
Extending this result to all graphs $H$ follows as an easy corollary. If $H$ has $h$ components $H_1,\dots H_h$, then $\hom(G,H)=\hom(G,H_1)+\dots+\hom(G,H_h)$ because $G$ is a connected graph. For $n$ sufficiently large, $G=K_{\delta,n-\delta}$ maximises $\hom(G,H_i)$ for each component $H_i$ and so $G=K_{\delta,n-\delta}$ also maximises $\hom(G,H)$. 

\begin{corollary}
\label{cor:connected}
For every $\delta\geq3$ and every graph $H$, there exists a constant $\kappa(\delta,H)$ such that the following holds: if $n\geq\kappa(\delta,H)$ and $G$ is a connected graph on $n$ vertices with minimum degree at least $\delta$, then we have $\hom(G,H)\leq\hom(K_{\delta,n-\delta},H)$. Further, if $H$ has a component which is neither a complete looped graph nor a complete balanced bipartite graph, we have equality if and only if $G=K_{\delta,n-\delta}$.
\end{corollary}
\noindent
We may identify a proper $q$-colouring of a graph $G$ with a graph homomorphism from $G$ into $K_q$. Therefore, counting the number of proper $q$-colourings of $G$ corresponds to counting the number of proper graph homomorphisms from $G$ into $K_q$. As $K_q$ is a connected graph, the following corollary also follows immediately from Theorem \ref{thm:connectedgraph}. This answers another question posed by Engbers \cite{E17}.

\begin{corollary}
Fix $\delta\geq3$ and $q>2$. Then, for $n$ sufficiently large, $K_{\delta,n-\delta}$ uniquely maximizes the number of proper $q$-colourings amongst all connected graphs on $n$ vertices with minimum degree at least $\delta$.
\end{corollary}

A natural extension to Corollary \ref{cor:connected} is to allow $G$ to have more than one component. Here the picture is less complete.

If $H$ is the graph consisting of a single edge with one of the vertices looped, then counting the number of $H$-colourings of a graph $G$ is equivalent to counting the number of independent sets in $G$. Extending previous work on this topic, Cutler and Radcliffe \cite{CR14} gave complete results for all values of $n$ and $\delta$. In particular, if $n\geq2\delta$, then $K_{\delta,n-\delta}$ is the unique graph which maximises $\hom(G,H)$.

Galvin \cite{G13} conjectured that, for any $H$, if $G$ was a $\delta$-regular graph on $n$ vertices, then $\hom(G,H)\leq\max\{\hom(K_{\delta,\delta},H)^{n/2\delta},\hom(K_{\delta+1},H)^{n/(\delta+1)}\}$. If this were true, it would mean that, whenever $2\delta(\delta+1)|n$, the $\delta$-regular graph on $n$ vertices which maximises the number of $H$-colourings is either $\frac n{2\delta}K_{\delta,\delta}$ or $\frac n{\delta+1}K_{\delta+1}$. Galvin's conjecture was shown to be false by Sernau \cite{S18}. He produced an infinite family of counterexamples as follows: fix $\delta$ and any simple loopless graph $H$ with no $(\delta+1)$-clique. Take any connected $\delta$-regular graph $G$ on $n<2\delta$ vertices with $\hom(G,H)>0$. He proved that there existed $k\in\mathbb{N}$ such that $\hom(G,kH)>\max\{\hom(K_{\delta+1},kH)^{n/(\delta+1)},\hom(K_{\delta,\delta},kH)^{n/2\delta}\}$ and hence that Galvin's conjecture was false.

Engbers \cite{E15} considered a similar question to Galvin but only when the order of $G$ was sufficiently large. He asked which graph on $n$ vertices with minimum degree $\delta$ maximises the number of $H$-colourings as the value of $n$ increases.

For general $H$ and $\delta=1$ or $\delta=2$, Engbers showed that $\hom(G,H)$ is maximised by one of $\frac{n}{\delta+1}K_{\delta+1}$, $\frac{n}{2\delta}K_{\delta,\delta}$ or $K_{\delta,n-\delta}$ (where the graph that maximises $\hom(G,H)$ depends on the structure of $H$). These results led him to make the following conjecture.

\begin{conjecture} [\cite{E15}]
\label{conj:main}
Fix $\delta\geq1$ and any graph $H$. Let $G$ be a graph on $n$ vertices with minimum degree at least $\delta$. There exists a constant $c(\delta,H)$ such that, for $n\geq c(\delta,H)$, we have
\[
\hom(G,H)\leq\max\big\{\hom(K_{\delta+1},H)^{\frac n{\delta+1}},\hom(K_{\delta,\delta},H)^{\frac n{2\delta}},\hom(K_{\delta,n-\delta},H)\big\}.
\]
\end{conjecture}
\noindent
In Section \ref{sec:counterexample}, we will use similar ideas to Sernau to construct counterexamples to Conjecture \ref{conj:main} whenever $\delta\geq3$.

On the other hand, we can show that Conjecture \ref{conj:main} does hold in certain circumstances. In Section \ref{sec:largedelta}, we will consider the case when the graph $H$ is fixed and $\delta$ and $n$ are sufficiently large. In particular, for each $k\in\mathbb{N}$, we consider the family $\mathcal{H}_k$ of all graphs with maximum degree $k$ that do not contain the complete looped graph on $k$ vertices or $K_{k,k}$ as a component. We will prove the following theorem.

\begin{theorem}
\label{thm:largedelta}
Fix any $k\in\mathbb{N}$. For every graph $H\in\mathcal{H}_k$ and every $\delta\geq\delta_0(H)$, the following holds: there exists a constant $n_0(\delta,H)$ such that, if $n\geq n_0(\delta,H)$ and $G$ is a graph on $n$ vertices with minimum degree $\delta$, then $\hom(G,H)\leq\hom(K_{\delta,n-\delta},H)$. Equality holds if and only if $G=K_{\delta,n-\delta}$.
\end{theorem}
\noindent
The graph $K_{\delta,n-\delta}$ need not maximise the number of $H$-colourings if $H$ has maximum degree $k$ and contains either the complete looped graph on $k$ vertices or $K_{k,k}$ as a component (i.e. $H\notin\mathcal{H}_k$). This is discussed in more detail in Section \ref{sec:conclusion}.

\begin{convention}
Throughout this paper, $G$ will be a simple graph without loops. We will adopt the same convention for vertex degrees as Engbers \cite{E17}: for any vertex $v\in V(H)$, we define $d(v)=|\{w\in V(H):vw\in E(H)\}|$. In particular, adding a loop to a vertex in $H$ increases the degree by one.
\end{convention}

\section{Proof of Theorem \ref{thm:connectedgraph}}
\label{sec:connectedgraph}
The following definition was introduced by Engbers \cite{E15}. We will use it in the proof of Theorem \ref{thm:connectedgraph} as well as in Section \ref{sec:largedelta}.

\begin{definition}
For any graph $H$ with maximum degree $k$ and $\delta\geq1$, we define $S(\delta,H)$ to be the set of vectors in $V(H)^{\delta}$ such that the elements of the vector have $k$ neighbours in common. We define $s(\delta,H)=|S(\delta,H)|$. As $H$ has at least one vertex of degree $k$, we have $s(\delta,H)\geq1$.
\end{definition}
\noindent
We will need the following theorem of Erd\H{o}s and P\'{o}sa.

\begin{theorem}[\cite{D10}]
\label{thm:erdos-posa}
There is a function $f:\mathbb{N}\rightarrow\mathbb{R}$ such that, given any $d\in\mathbb{N}$, every graph contains either $d$ disjoint cycles or a set of at most $f(d)$ vertices meeting all its cycles.
\end{theorem}
\noindent
We will frequently use the following lemma of Engbers.

\begin{lemma}[\cite{E15}] 
\label{lem:pathcounting}
Suppose $H$ is not the complete looped graph on $k$ vertices or $K_{k,k}$. Then, for any two vertices $i$, $j$ of $H$ and for $r\geq4$, there are at most $(k^2-1)k^{r-4}$ $H$-colourings of $P_r$ that map the initial vertex of that path to $i$ and the terminal vertex to $j$.
\end{lemma}
\noindent
We will also need the following simple observation.

\begin{proposition}
\label{prop:colouring}
Let $G$ and $H$ be graphs with $G$ connected and $X\subseteq V(G)$. Suppose the vertices of $X$ have already been mapped to vertices of $H$. The remaining vertices of $G$ can be mapped into $V(H)$ in such a way that there are at most $\Delta(H)$ choices for each vertex of $V(G)\backslash X$.
\end{proposition}

\begin{proof}
Because $G$ is connected, there is a path from each vertex of $V(G)\backslash X$ to $X$. We order the vertices of $V(G)\backslash X$ by increasing distance from $X$. Each vertex $v\in V(G)\backslash X$ either has a neighbour in $X$ or a neighbour before it in the ordering. Therefore, when we come to colour $v$, one of its neighbours has already been coloured so there are at most $\Delta(H)$ choices for $v$.
\end{proof}

\begin{proof}[Proof of Theorem \ref{thm:connectedgraph}]
Let $\delta\geq3$ be fixed and let $H$ be a connected graph with maximum degree $k\in\mathbb{N}$. We have $|V(H)|\geq k$. There are two special cases to look at before we consider a general $H$.
\begin{enumerate}
\item \textit{$H$ is the complete looped graph on $k$ vertices.}
\\
If $G$ is any graph on $n$ vertices, we find that $\hom(G,H)=k^n$ because any vertex of $G$ can be mapped to any vertex of $H$. Hence, as any graph on $n$ vertices with minimum degree $\delta$ maximises the number of $H$-colourings, we have $\hom(G,H)\leq\hom(K_{\delta,n-\delta},H)$ as required.

\item \textit{$H=K_{k,k}$.}
\\
$H$ is bipartite so $\hom(G,H)\neq0$ if and only if $G$ is bipartite. For any connected bipartite graph $G$ on $n$ vertices, $\hom(G,H)=2k^n$. This means that any connected bipartite graph on $n$ vertices with minimum degree $\delta$ maximises the number of $H$-colourings and hence  $\hom(G,H)\leq\hom(K_{\delta,n-\delta},H)$ as required.
\end{enumerate}
As the theorem is true in these two cases, we may assume that $H$ is not the complete looped graph on $k$ vertices or $K_{k,k}$. We may also assume that $k\geq2$ as we have already dealt with the cases when $H$ is a single looped vertex and when $H=K_{1,1}$. Hence we may apply Lemma \ref{lem:pathcounting} when required.

Let $G$ be a graph on $n$ vertices with minimum degree $\delta$ that has the maximum number of $H$-colourings. We know that $H$ has at least one vertex $v$ of degree $k$. When considering $H$-colourings of $K_{\delta,n-\delta}$, we can map the vertex class of size $\delta$ to $v$ and the other vertex class to the neighbours of $v$. Hence, $\hom(K_{\delta,n-\delta},H)\geq k^{n-\delta}$. 

We will proceed to determine the structure of $G$. The assumption that $G$ has most $H$-colourings tells us that $\hom(G,H)\geq\hom(K_{\delta,n-\delta},H)$. We will show that, for $n$ sufficiently large, we must have $G=K_{\delta,n-\delta}$.
\\
\\
\textit{Claim 1: $G$ has a bounded number of disjoint cycles.}\\
Suppose that $G$ has $d$ disjoint cycles. We colour $G$ in the following way. Pick any vertex of $G$ and map it to any vertex of $H$. Take a shortest path from the starting vertex to a vertex on one of the disjoint cycles. There are at most $k$ ways to map each vertex on this path to vertices of $H$. We then consider the other vertices on the cycle (as the end vertex of the path has already been mapped to a vertex of $H$). Lemma \ref{lem:pathcounting} gives at most $(k^2-1)k^{t-3}$ ways to map these vertices to $H$, where $t$ is the number of vertices in the cycle. We then repeat this process of finding a shortest path from the already mapped vertices to one of the disjoint cycles and mapping the vertices in the path and cycle to $H$. Once all of the vertices in disjoint cycles have been considered, any remaining vertices can be mapped greedily with at most $k$ choices for each by Proposition \ref{prop:colouring}. Therefore
\[
\hom(G,H)\leq|V(H)|(k^2-1)^dk^{n-2d-1}<|V(H)|k^{n-1}e^{-\frac{d}{k^2}}.
\]
This is strictly smaller than $k^{n-\delta}$ whenever $d>k^2\log|V(H)|+k^2(\delta-1)\log k$. As $\hom(G,H)$ is maximal, it follows that $G$ has bounded number of disjoint cycles. This bound only depends on $H$ and $\delta$. Hence we have proved the claim.
\\
\\
Applying Theorem \ref{thm:erdos-posa} to $G$, we find that there exists a constant $\alpha=\alpha(\delta,H)$ such that $G$ can be made acyclic by removing at most $\alpha$ vertices. We can therefore partition the vertices of $G$ into a set $A$ of size at most $\alpha$ and a set $F$ such that $G[F]$ is a forest.

We will show that we can make $F$ into an independent set by moving at most a constant number of vertices from $F$ to $A$. This constant depends only on $\delta$ and $H$ and not on the number of vertices in $G$.

We say that a component of a graph is \textit{non-trivial} if it contains at least one edge.
\\
\\
\textit{Claim 2: The forest $F$ has a bounded number of non-trivial components.}\\
Suppose $F$ has $a$ non-trivial components, $G_1,\dots G_a$. Each $G_i$ is a tree and so contains a maximal path $P_i$. As every vertex in $G$ has degree at least $\delta\geq3$, each end-vertex of $P_i$ must have a neighbour in $A$. We colour $G$ in the following way. First map $A$ into $H$. There are at most $|V(H)|^{|A|}$ ways to do this. We then consider each $G_i$ in turn. By Lemma 2.2, there are at most $(k^2-1)k^{|P_i|-2}$ ways to colour $P_i$ and at most $k$ ways to colour each of the other vertices of $G_i$. Finally, we consider the remaining vertices of $G$, each of which has at most $k$ possible choices by Proposition \ref{prop:colouring}. Hence
\[
\hom(G,H)\leq|V(H)|^{|A|}(k^2-1)^ak^{n-|A|-2a}<|V(H)|^{\alpha}k^{n-\alpha}e^{-\frac{a}{k^2}}.
\]
This is strictly less than $k^{n-\delta}$ whenever $a>k^2\alpha\log|V(H)|+k^2(\delta-\alpha)\log k$. The maximality of $\hom(G,H)$ means that there exists a constant depending only on $\delta$ and $H$ that bounds the number of non-trivial components of $F$ and hence proves the claim.
\\
\\
Let $T$ be any non-trivial component of $F$. Define $T'$ to be the subtree obtained from $T$ by deleting all of the leaves. We will show that the size of $T'$ is bounded by a constant that only depends on $\delta$ and $H$. This is done in two steps: first we show that the maximal length of a path in $T$ is bounded and then we show that $T'$ can only have a bounded number of leaves. Together, these two claims bound the size of $T'$.
\\
\\
\textit{Claim 3: The length of the longest path in $T$ is bounded.}\\
Suppose the longest path $P$ in $T$ is $u_1v_1u_2v_2\dots$ and has length $b$. We may write $b=2b'+r$ where $r\in\{0,1\}$. The minimum degree of $G$ is at least $\delta\geq3$ and $T$ is acyclic. Therefore, each vertex of $P$ has a neighbour which is not on $P$. Further, every leaf of $T$ must have a neighbour in $A$.

We colour the vertices of $G$ as follows. First, colour $A$. Next, we colour the vertices of $P$ using the following algorithm. Initially, $i=1$. The algorithm colours vertices $u_i$ and $v_i$ at step $i$ (and possibly some other vertices of $T$ that do not lie on $P$).

\begin{figure}[h]
\centering
\begin{tikzpicture}
\draw (-1,-0.5) ellipse (0.6cm and 2cm) node {$A$};
\node[main node, label={0:$u_1$}] (u) at (3,1.5) {};
\node[main node, label={0:$v_1$}] (v) at (3,0.5) {};
\node[main node, label={10:$P$}] (w) at (3,-2.5) {};
\path[draw,thick]
	(u) edge node {} (v)
	(v) edge node[below, rotate=90] {...} (w)
	(u) edge node {} (-1,1.1)
	(v) edge node {} (-1,0.5);
\draw[blue] (3,1) ellipse (0.7cm and 1cm);
\draw[red] (-1,-0.5) ellipse (0.8cm and 2.2cm);

\draw (7,-0.5) ellipse (0.6cm and 2cm) node {$A$};
\node[main node, label={0:$u_1$}] (u) at (11,1.5) {};
\node[main node, label={10:$v_1$}] (v) at (11,0.5) {};
\node[main node, label={10:$P$}] (w) at (11,-2.5){};
\node[main node, label={270:$Q_1$}] (x) at (9,-1.5) {};
\path[draw,thick]
	(u) edge node {} (v)
	(v) edge node[below, rotate=90] {...} (w)
	(u) edge node {} (7,1.1)
	(x) edge node {} (7,0.5);
\draw[decorate, decoration=snake] (v)--++(x);
\draw[red] (7,-0.5) ellipse (0.8cm and 2.2cm);
\draw [blue] plot [smooth cycle, tension=2] coordinates {(11,1.9) (11,-0.3) (8.6,-1.7)};

\begin{customlegend}[legend cell align=left,legend entries={vertices already coloured, vertices coloured in this step}, legend style={at={(5.5,-3.5)},anchor=north,font=\footnotesize}]
\addlegendimage{no markers,red}
\addlegendimage{no markers,blue}
\end{customlegend}
\end{tikzpicture}
\caption{On the left, $v_1$ has a neighbour in $A$; on the right, $v_1$ does not.}
\label{fig:longpathstart}
\end{figure}
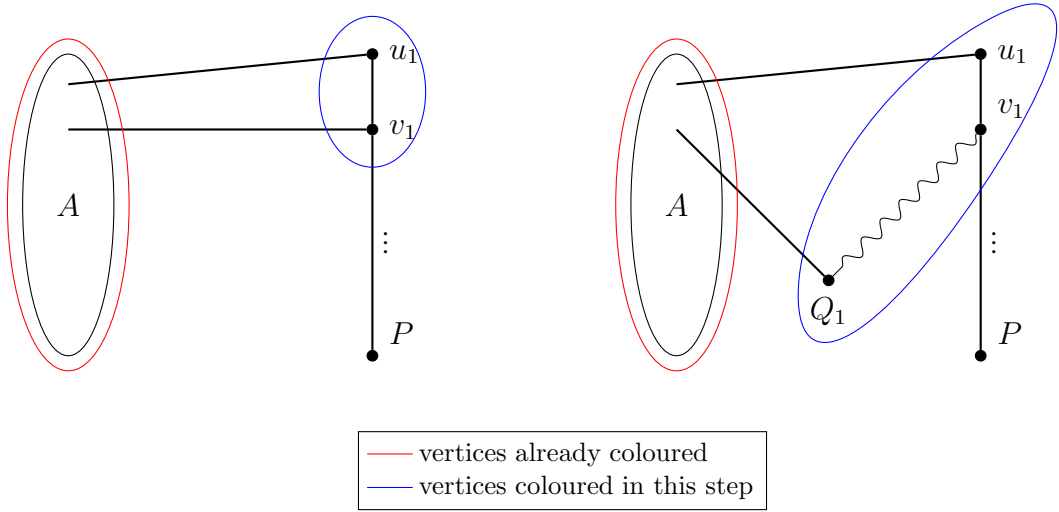

At the $i^\text{th}$ step, consider vertices $u_i$ and $v_i$ on $P$. If $i=1$, $u_i$ is an end-vertex of $P$ and so has a neighbour in $A$; if $i\neq1$, $u_i$ has $v_{i-1}$ as a neighbour. Hence, we know $u_i$ is adjacent to a vertex which has already been coloured. Consider the vertex $v_i$. If $v_i$ has a neighbour in A, we have a path of length 4 starting and ending at vertices which have already been coloured. Lemma \ref{lem:pathcounting} tells us there are at most $k^2-1$ choices for $u_i$ and $v_i$ (see Figure \ref{fig:longpathstart}). If $v_i$ does not have a neighbour in $A$, it must have another neighbour in $T$ which does not lie on $P$. Take a maximal path $Q_i$ in $T$, which starts at $v_i$ and avoids $P$. The end-vertex of $Q_i$ that is not $v_i$ must be a leaf in $T$ and hence has a neighbour in $A$ (see Figure \ref{fig:longpathstart}). We therefore have a path of length $|Q_i|+3$ which starts and ends with vertices that have already been coloured and has $u_i\cup Q_i$ as the internal vertices. Lemma 2.2 gives at most $(k^2-1)k^{|Q_i|-1}$ ways to colour the path $u_i\cup Q_i$. We then proceed to the $(i+1)^\text{th}$ step of the algorithm.

After $b'$ steps, we have coloured $2b'$ vertices of $P$ (and possibly some other vertices of $T$). We finish by colouring all of the remaining vertices of $G$, each of which has at most $k$ choices by Proposition \ref{prop:colouring}. Therefore
\[
\hom(G,H)\leq|V(H)|^{|A|}(k^2-1)^{b'}k^{n-|A|-2b'}<|V(H)|^{\alpha}k^{n-\alpha}e^{-\frac{b'}{k^2}}.
\]
This is strictly less than $k^{n-\delta}$ whenever $b'>k^2\alpha\log|V(H)|+k^2(\delta-\alpha)\log k$. Because $\hom(G,H)$ is maximal, there exists a constant depending only on $\delta$ and $H$ which bounds the length of a maximal path in any non-trivial component of $F$ as required.
\\
\\
\textit{Claim 4: $T'$ has a bounded number of leaves.}\\
Suppose $T'$ has $l$ leaves. Each leaf of $T'$ has at least two neighbours which are not in $T'$ because the minimum degree of $G$ is at least $\delta\geq3$. At least one of these neighbours is a leaf of $T$. Similarly, every leaf of $T$ has a neighbour in $A$.

We colour $G$ by first colouring the vertices of $A$. For each leaf $v$ of $T'$, there are two possibilities. If $v$ has two neighbours $u$ and $w$ which are leaves of $T$, there is a path of length 5 with end vertices in $A$ and internal vertices $u$, $v$ and $w$. By Lemma \ref{lem:pathcounting} there are at most $(k^2-1)k$ ways to colour the path $uvw$. If $v$ only has one neighbour $u$ which is a leaf of $T$, then $v$ must also have a neighbour in $A$ because it has at least $\delta$ neighbours and only one of these can be in $T'$ (see Figure \ref{fig:treeleaves}). Apply Lemma \ref{lem:pathcounting} to the path with end vertices in $A$ and internal vertices $u$ and $v$. There are at most $k^2-1$ choices for the colours of $u$ and $v$.

\begin{figure}[h]
\centering
\begin{tikzpicture}
\draw (-1.1,-0.5) ellipse (0.6cm and 2cm) node {$A$};
\node[main node,label={270:$u$}] (u) at (0.9,0.5) {};
\node[main node,label={90:$w$}] (w) at (0.9,-1.5) {};
\node[main node,label={180:$v$}] (v) at (1.9,-0.5) {};
\draw (2.7,-0.5) ellipse (1.4cm and 0.6cm) node {$T'$};
\path[draw,thick]
	(u) edge node {} (v)
	(w) edge node {} (v)
	(v) edge node {} (2.3,-0.5)
	(u) edge node {} (-1.1,0.1)
	(w) edge node {} (-1.1,-1.1);
\draw[red] (-1.1,-0.5) ellipse (0.8cm and 2.2cm);
\draw [blue] plot [smooth cycle, tension=2] coordinates {(0.7,0.7) (0.7,-1.7) (2.1,-0.5)};

\draw (6.9,-0.5) ellipse (0.6cm and 2cm) node {$A$};
\node[main node,label={90:$u$}] (u) at (8.9,-0.5) {};
\node[main node,label={90:$v$}] (v) at (9.9,-0.5) {};
\draw (10.7,-0.5) ellipse (1.4cm and 0.6cm) node {$T'$};
\path[draw,thick]
	(u) edge node {} (v)
	(v) edge node {} (10.3,-0.5)
	(u) edge node {} (6.9,0)
	(v) edge[bend left] node {} (6.9,-1);
\draw[red] (6.9,-0.5) ellipse (0.8cm and 2.2cm);
\draw[blue] (9.4,-0.5) circle (0.8cm);

\begin{customlegend}[legend cell align=left,legend entries={vertices already coloured, vertices coloured in this step}, legend style={at={(5.5,-3.5)},anchor=north,font=\footnotesize}]
\addlegendimage{no markers,red}
\addlegendimage{no markers,blue}
\end{customlegend}
\end{tikzpicture}
\caption{On the left, $v$ has two leaves as neighbours; on the right, $v$ has one.}
\label{fig:treeleaves}
\end{figure}
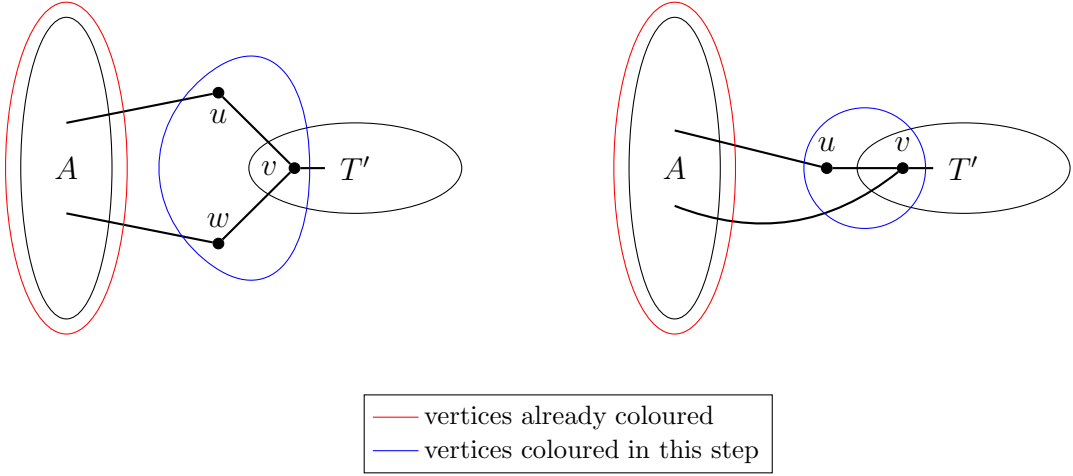

Once each leaf of $T'$ has been assigned to a vertex of $H$, there are at most $k$ choices for each of the remaining vertices of $G$ by Proposition \ref{prop:colouring}. Therefore
\[
\hom(G,H)\leq|V(H)|^{|A|}(k^2-1)^lk^{n-|A|-2l}<|V(H)|^{\alpha}k^{n-\alpha}e^{-\frac{l}{k^2}}.
\]
This is strictly less than $k^{n-\delta}$ whenever $l>k^2\alpha\log|V(H)|+k^2(\delta-\alpha)\log k$. The maximality of $\hom(G,H)$ means that the maximum number of leaves $T'$ can have is bounded above by a constant depending only on $\delta$ and $H$ as required.
\\
\\
Claims 3 and 4 show that, for each non-trivial component $T$ of $F$, the subtree $T'$ consisting of $T$ without its leaves has maximal size bounded by a constant $t(\delta,H)$. Claim 2 shows that there are at most $a(\delta,H)$ non-trivial components of $F$ for some constant $a(\delta,H)$.

We can make $F$ into an independent set by moving some (possibly all) of the vertices of each $T'$ from $F$ to $A$. If any non-trivial component has $T'=\emptyset$, then $T$ is a single edge and in this case we just move one of the end vertices from $F$ to $A$. Hence, by moving at most $a(\delta,H)t(\delta,H)$ vertices from $F$ to $A$, we can turn the forest into an independent set.

We have now partitioned the vertices of $G$ into sets of vertices $L$ and $R$ where $|L|\leq \alpha(\delta,H)+a(\delta,H)t(\delta,H)$ and $R$ is an independent set. The size of $L$ is bounded above by a constant that only depends on $\delta$ and $H$; it does not depend on the size of $G$.

Each vertex in $R$ has at least $\delta$ neighbours in $L$ because of the minimum degree of the vertices in $G$. By the pigeonhole principle, there exists a set $Y\subseteq L$ of size $\delta$ such that $Y$ is contained in the neighbourhood of at least $(n-|L|)/{{|L|}\choose{\delta}}\geq cn$ vertices of $R$ for some constant $c=c(\delta,H)$. Hence, $G$ contains the subgraph $K_{\delta,cn}$.

If $G$ does not contain $K_{\delta,n-\delta}$ as a subgraph, then $Y$ is not a dominating set for $G$. Therefore, the subgraph induced by $G\backslash Y$ has a non-trivial component. If $G\backslash Y$ contains a non-trivial tree, take a maximal path $X$ in this tree. Otherwise, choose $X$ to be a cycle together with a shortest path from the cycle to $Y$.

We may colour the vertices of $G$ in such a way that $Y$ is always coloured first. Recall the definition of $S(\delta,H)$ given at the beginning of Section \ref{sec:connectedgraph}.

If $Y$ is coloured using a vector from $S(\delta,H)$, we then colour the vertices of $X$. There are at most $(k^2-1)k^{|X|-2}$ ways do this. Finally, we colour the remaining vertices, each of which has at most $k$ choices by Proposition \ref{prop:colouring}. This gives at most $s(\delta,H)(k^2-1)k^{n-\delta-2}$ such colourings.

Alternatively, if $Y$ is not coloured using a vector from $S(\delta,H)$, then there are at most $k-1$ ways to map each of the other $cn$ vertices of the $K_{\delta,cn}$ subgraph into $H$. There are then at most $k$ choices for each of the remaining vertices of $G$ by Proposition \ref{prop:colouring}. There are at most $|V(H)|^\delta(k-1)^{cn}k^{n-\delta-cn}$ such colourings.

Combining the above gives
\begin{align*}
\hom(G,H)
&\leq s(\delta,H)(k^2-1)k^{n-\delta-2}+|V(H)|^\delta(k-1)^{cn}k^{n-cn-\delta}\\
&=s(\delta,H)k^{n-\delta}-s(\delta,H)k^{n-\delta-2}+|V(H)|^\delta(k-1)^{cn}k^{n-cn-\delta}\\
&<s(\delta,H)k^{n-\delta}
\end{align*}
for sufficiently large values of $n$.

If $G$ contains $K_{\delta,n-\delta}$ as a subgraph and $G\neq K_{\delta,n-\delta}$, then we know that $G$ contains at least one extra edge between two vertices in the same partition class. Clearly, every mapping of $G$ into $H$ is also a mapping of $K_{\delta,n-\delta}$ into $H$. We will show below that the converse is not true.

If $ij$ is an edge in $H$, then mapping the size $\delta$ partition class of $K_{\delta,n-\delta}$ to $i$ and the other partition class to $j$ is a proper mapping of $K_{\delta,n-\delta}$ into $H$. However, it is only a proper mapping of $G$ to $H$ if the partition class containing the extra edge is mapped to a looped vertex. Therefore, if $H$ has a non-looped vertex, $\hom(G,H)<\hom(K_{\delta,n-\delta})$.

Suppose every vertex of $H$ is looped. We assumed that $H$ was connected and not the complete looped graph so there will be non-adjacent vertices $j$ and $k$ which have a common neighbour $i$. We may map the partition class with the extra edge to vertices $j$ and $k$ and the other partition class to $i$. If the extra edge has one endpoint in $j$ and the other in $k$, we do not get a proper $H$-colouring of $G$ but it is a valid $H$-colouring of $K_{\delta,n-\delta}$. Hence $\hom(G,H)<\hom(K_{\delta,n-\delta})$.

Therefore, if $\hom(G,H)$ is maximal and $n$ is sufficiently large, then we must have $G=K_{\delta,n-\delta}$.
\end{proof}

\section{Counterexample to Conjecture \ref{conj:main}}
\label{sec:counterexample}

We write $T_t(x)$ for the \textit{$t$-partite Tur\'{a}n graph on $x$ vertices} (i.e. the complete $t$-partite graph on $x$ vertices with the vertex classes as equal as possible).

For every $\delta\geq3$, we will construct a graph $H$ such that, for infinitely many values of $n$, the number of $H$-colourings is uniquely maximised by a disjoint union of complete multipartite graphs. This shows that Conjecture \ref{conj:main} does not hold. For simplicity, we first assume that $(t-1)|\delta$ for some $3\leq t\leq\delta$.

\begin{theorem}
\label{thm:counter}
Fix $\delta\geq3$ and $3\leq t\leq\delta$ such that $\delta=(t-1)\alpha$ for some $\alpha\in\mathbb{N}$. Then there exists a constant $k_0(\delta)$ such that the following holds for all values of $m\in\mathbb{N}$: if $k\geq k_0(\delta)$ and $G$ is any graph on $n=mt\alpha$ vertices with minimum degree at least $\delta$, then we have $\hom(G,kK_t)\leq\hom(mT_t(t\alpha),kK_t)$ with equality if and only if $G=mT_t(t\alpha)$.
\end{theorem}

\begin{proof}
Fix $\delta\geq3$ and $3\leq t\leq\delta$ as above where $\delta=(t-1)\alpha$. Take $k$ sufficiently large that $(t!k)^{1/(t\alpha)}>tk^{1/(t\alpha+1)}$.

Clearly, $\hom(K_{t+1},kK_t)=0$ and so we only need to consider graphs which are $K_{t+1}$-free.

Any $K_{t+1}$-free graph with minimum degree at least $\delta$ has at least $t\alpha$ vertices. Tur\'{a}n's theorem tells us that $T_t(t\alpha)$ is the only such graph with exactly $t\alpha$ vertices. It is easy to see that $\hom(T_t(t\alpha),kK_t)=t!k$.

Let $m\in\mathbb{N}$ and take $G$ to be any graph on $n=mt\alpha$ vertices with minimum degree at least $\delta$. We may assume that $G$ has $a$ components $G_1,\dots G_a$ with $|G_1|\geq\dots\geq|G_a|\geq t\alpha$. Then $\hom(G,kK_t)=\prod_{i=1}^a\hom(G_i,kK_t)$. If $|G_1|=t\alpha$, then $|G_i|=t\alpha$ for all $i$ and hence $G=mT_t(t\alpha)$.

Suppose that $|G_1|>t\alpha$. We know that, if $|G_i|=t\alpha$, then $G_i=T_t(t\alpha)$ and $\hom(G_i,kKt)=t!k$. If $|G_i|>t\alpha$, then we may colour the vertices of $G_i$ greedily to get $\hom(G_i,kK_t)\leq tk(t-1)^{|G_i|-1}<kt^{|G_i|}$. We chose $k$ such that $(t!k)^{1/(t\alpha)}>tk^{1/(t\alpha+1)}$. Using this and the fact that $|G_i|\geq t\alpha+1$, we have $\hom(G_i,kK_t)<(t!k)^{|G_i|/(t\alpha)}$. Combining these two observations, we get
\[
\hom(G,kK_t)=\prod_{i=1}^a\hom(G_i,kK_t)<(t!k)^{n/(t\alpha)}=(t!k)^m=\hom(mT_t(t\alpha),kK_t).
\]
Therefore, if $G$ is any graph on $n=mt\alpha$ vertices with minimum degree at least $\delta$, we have $\hom(G,kK_t)\leq\hom(mT_t(t\alpha),kK_t)$. We have equality if and only if $G=mT_t(t\alpha)$.
\end{proof}
\noindent
We may use the techniques above to show that, if $(t-1)|(\delta +1)$, then a similar result holds -- there is a graph $H$ such that the number of $H$-colourings is uniquely maximised by a union of complete $t$-partite graphs. Therefore, for every $\delta\geq3$, by taking $t=3$, we can produce a counterexample to Conjecture \ref{conj:main}.

In all of the examples we have seen so far, the number of $H$-colourings has been maximised by the union of complete multipartite graphs. We will now give an example where this is not the case.

Take $\delta=7$ and $t=4$ and choose $k$ as in Theorem \ref{thm:counter}. Let $H=kK_4$, $m\in\mathbb{N}$ and take $G$ to be any graph on $n=10m$ vertices with minimum degree at least $7$. As before, we may assume that $G$ is $4$-colourable. If $G$ has a component with at least $11$ vertices, then we can show, in a similar way to Theorem \ref{thm:counter}, that $\hom(G,kK_4)<\hom(mT_4(10),kK_4)$. Any union of complete multipartite graphs except $mT_4(10)$ is either not 4-colourable or contains a component with at least 11 vertices. Therefore, $mT_4(10)$ maximises the number of $H$-colourings among unions of complete multipartite graphs. However, the number of $H$-colourings is not maximised overall by $mT_4(10)$. Let $T'$ be the graph formed from $T_4(10)$ by removing a perfect matching between the two vertex classes of size $2$. Then $\hom(mT',kK_4)=2\hom(mT_4(10),kK_4)$.

\section{Proof of Theorem \ref{thm:largedelta}}
\label{sec:largedelta}
We will need the following simple observation.

\begin{proposition}
\label{prop:disjointcycles}
Fix $d\in\mathbb{N}$. Let $G$ be any graph with minimum degree at least $3d$. Then $G$ has at least $d$ disjoint cycles.
\end{proposition}

\begin{proof}
If $d=1$, the minimum degree of $G$ is at least 3 and so $G$ contains a cycle. If $d>1$, take $C$ to be a shortest cycle in $G$. Each vertex in $G$ has at most 3 neighbours on $C$ or else we would be able to find a shorter cycle. Removing the vertices in $C$ reduces the minimum degree by at most 3. Therefore, by induction, we can find at least $d-1$ disjoint cycles in $G\backslash V(C)$.
\end{proof}
\noindent
Before proving Theorem \ref{thm:largedelta}, we will prove a couple of useful lemmas. Recall the definitions of $S(\delta,H)$ and $s(\delta,H)$ given at the start of Section \ref{sec:connectedgraph}.

\begin{lemma}
\label{lem:boundlargebipartite}
Fix $\delta\geq1$ and $k\geq2$. Fix $H$ to be any graph with maximum degree $k$. Then there exists a constant $\beta(\delta,H)$ such that, for $n\geq\beta(\delta,H)$, we have $\hom(K_{\delta,n-\delta},H)\leq s(\delta,H)k^{n+1-\delta}$.
\end{lemma}

\begin{proof}
The graph $K_{\delta,n-\delta}$ has two vertex classes. Denote the class of size $\delta$ by $Z$. When we are counting the number of $H$-colourings of $K_{\delta,n-\delta}$, we will colour vertices in $Z$ first and then the remaining vertices may be coloured greedily. There are two possibilities: either $Z$ is coloured so that all of the vertices used in $H$ have $k$ common neighbours (i.e. we use a vector from $S(\delta,H)$) or the vertices in $H$ used to colour $Z$ have strictly fewer than $k$ neighbours in common.

First, we consider the case where $Z$ is coloured using a vector from $S(\delta,H)$. When we come to colour the vertices of $G\backslash Z$, there are exactly $k$ choices for each one. Therefore, there are exactly $s(\delta,H)k^{n-\delta}$ such colourings.

Next, we consider the case where $Z$ is coloured so that the vertices used do not have $k$ common neighbours in $H$. This leaves at most $k-1$ ways to map the vertices of $G\backslash Z$ into $H$. Hence, there are at most $|V(H)|^{\delta}(k-1)^{n-\delta}$ such colourings.

Combining the above gives
\[
\hom(K_{\delta,n-\delta},H)\leq s(\delta,H)k^{n-\delta}+|V(H)|^{\delta}(k-1)^{n-\delta}.
\]
Hence, for $n$ sufficiently large, we have
\begin{align*}
\hom(K_{\delta,n-\delta},H)&\leq s(\delta,H)k^{n-\delta}+k^{n-\delta}\\
&\leq s(\delta,H)k^{n+1-\delta}.
\end{align*}
This proves the required result.
\end{proof}

\begin{lemma}
\label{lem:boundlargedelta}
Fix $H$ to be any graph with maximum degree $k\in\mathbb{N}$ that does not have the complete looped graph on $k$ vertices or $K_{k,k}$ as a component. There exists a constant $\delta_0(H)$ such that, if $\delta\geq\delta_0(h)$ and $G$ is a connected graph on $n$ vertices with minimum degree $\delta$, then $\hom(G,H)<k^{n-1}$.
\end{lemma}

\begin{proof}
The minimum degree condition on $G$ ensures that $n\geq\delta+1$. The restrictions on $H$ mean that $k\geq2$.

Let $H$ have $h$ components $H_1,\dots H_h$. As $G$ is connected, any $H$-colouring of $G$ maps $G$ to a single component $H_i$ and so $\hom(G,H)=\sum_{i=1}^h\hom(G,H_i)$. We therefore first count the number of $H_i$-colourings of $G$ for each $i\in[h]$. There are three cases to consider.
\\
\\
\textit{Case 1.} Let $H_i$ be a complete looped graph on $l$ vertices where $l<k$. Then $\hom(G,H_i)=l^n\leq(k-1)^n$. This is strictly less than $k^{n-h-1}$ whenever $n>\frac{(h+1)\log k}{\log k-\log(k-1)}$.
\\
\\
\textit{Case 2.} Let $H_i=K_{l,l}$ where $l<k$. Then $\hom(G,H_i)=2l^n\leq2(k-1)^n$. This is strictly less than $k^{n-h-1}$ whenever $n>\frac{\log2+(h+1)\log k}{\log k-\log(k-1)}$. 
\\
\\
\textit{Case 3.} Let $H_i$ be any connected graph which is not the complete looped graph on $l$ vertices or $K_{l,l}$ for some $l\leq k$. Suppose $G$ has $d$ vertex disjoint cycles $C_1,\dots C_d$. We colour $G$ in the following way:
\begin{enumerate}
\item Pick any vertex of $G$ and map it to any vertex of $H_i$.
\item Find a shortest path $P$ from the already coloured vertices of $G$ to an uncoloured vertex on one of the cycles $C_j$. There are at most $k$ ways to map each vertex on this path to vertices of $H_i$.
\item The end vertex of $P$ has already been mapped to a vertex of $H_i$ so we consider the other vertices on the cycle $C_j$. Lemma \ref{lem:pathcounting} gives at most $(k^2-1)k^{|C_j|-3}$ ways to map these vertices to $H_i$.
\item If, for some $j'\in\{1,\dots d\}$, the cycle $C_{j'}$ has not yet been coloured, go back to step 2.
\item Colour any remaining uncoloured vertices in a greedy fashion. By Proposition \ref{prop:colouring}, there are at most $k$ choices for each vertex.
\end{enumerate}
By colouring $G$ in this way, we find that
\[
\hom(G,H_i)\leq|V(H_i)|(k^2-1)^dk^{n-2d-1}<|V(H_i)|k^{n-1}e^{-\frac d{k^2}}.
\]
This is strictly less than $k^{n-h-1}$ whenever $d>k^2\log|V(H_i)|+k^2h\log k$.
\\
\\
Choose $\delta\geq\max\big\{3k^2\log|V(H)|+3k^2h\log k,\frac{(h+1)\log k}{\log k-\log(k-1)}\big\}$ and note that $n\geq\delta+1$. If $H_i$ is in either Case 1 or Case 2, then $n$ is large enough that $\hom(G,H_i)<k^{n-h-1}$. If $H_i$ is in Case 3, then, by Proposition \ref{prop:disjointcycles}, we have that the number of disjoint cycles in $G$ is at least $k^2\log|V(H)|+k^2h\log k$ and hence $\hom(G,H_i)<k^{n-h-1}$. Then
\[
\hom(G,H)=\sum_{i=1}^{h}\hom(G,H_i)<hk^{n-h-1}<k^{n-1}.
\]
Hence, if $H$ does not contain the complete looped graph on $k$ vertices or $K_{k,k}$ as a component, we have $\hom(G,H)<k^{n-1}$ for $\delta$ sufficiently large as required. 
\end{proof}
\noindent
We are now ready to prove the main result.

\begin{proof}[Proof of Theorem \ref{thm:largedelta}]
Let $H$ be any graph with maximum degree $k$ that does not have the complete looped graph on $k$ vertices or $K_{k,k}$ as a component. This allows us to apply Lemma \ref{lem:boundlargedelta} as required.

Choose $\delta\geq\delta_0(H)$ where $\delta_0(H)$ is the constant found in Lemma \ref{lem:boundlargedelta}. Set $\lambda(\delta,H)=\max\{\kappa(\delta,H), \beta(\delta,H)\}$ where $\kappa(\delta,H)$ is the constant found in Theorem \ref{thm:connectedgraph} and $\beta(\delta,H)$ is the constant found in Lemma \ref{lem:boundlargebipartite}. Now, choose $n>(\delta-1)(\lambda(\delta,H)-1)$.

Let $G$ be a graph on $n$ vertices with minimum degree $\delta$ that has the maximum number of $H$-colourings. Clearly, $\hom(G,H)\geq\hom(K_{\delta,n-\delta},H)\geq s(\delta,H)k^{n-\delta}\geq k^{n-\delta}$.

Let $G$ have $t$ components $G_1,\dots G_t$. An $H$-colouring of $G$ comprises of separate $H$-colourings of each component $G_i$ and therefore $\hom(G,H)=\prod_{i=1}^t\hom(G_i,H)$. As $G$ has the most $H$-colourings among all graphs on $n$ vertices with minimum degree $\delta$, we must also have that $G_i$ has the most $H$-colourings among all graphs on $|G_i|$ vertices with minimum degree $\delta$ for each $i\in\{1,\dots t\}$.
\\
\\
\noindent
\textit{Claim 1: $G$ has a bounded number of components.}\\
By Lemma \ref{lem:boundlargedelta}, we have that $\hom(G_i,H)<k^{|G_i|-1}$ for each $i\in\{1,\dots t\}$ so
\[
\hom(G,H)=\prod_{i=1}^t\hom(G_i,H)<\prod_{i=1}^tk^{|G_i|-1}=k^{n-t}.
\]
If $t\geq\delta$, then we have $\hom(G,H)<k^{n-\delta}\leq\hom(K_{\delta,n-\delta},H)$ and this contradicts our assumption that $G$ has the maximum number of $H$-colourings.
\\
\\
\noindent
Hence we know that $G$ has at most $\delta-1$ components. By the pigeonhole principle, there is a component of $G$ with at least $\lambda(\delta,H)$ vertices. Without loss of generality, we may assume this component is $G_1$. By Theorem \ref{thm:connectedgraph}, we have that $G_1=K_{\delta,|G_1|-\delta}$ and, applying Lemma \ref{lem:boundlargebipartite}, we find that $\hom(G_1,H)\leq s(\delta,H)k^{|G_1|+1-\delta}$.
\\
\\
\noindent
\textit{Claim 2: $G$ has exactly one component.}\\
Suppose $t>1$. We know $\hom(G_1,H)\leq s(\delta,H)k^{|G_1|+1-\delta}$. By Lemma \ref{lem:boundlargedelta}, we have $\hom(G_2,H)<k^{|G_2|-1}$. Hence
\begin{align*}
\hom(G_1\cup G_2,H)&<s(\delta,H)k^{|G_1|+1-\delta}k^{|G_2|-1}\\
&=s(\delta,H)k^{|G_1|+|G_2|-\delta}\\
&\leq\hom(K_{\delta,|G_1|+|G_2|-\delta},H).
\end{align*}
Replacing $G_1\cup G_2$ by $K_{\delta,|G_1|+|G_2|-\delta}$ increases the number of $H$-colourings of $G$, which contradicts our assumption that $G$ has the maximum number of $H$-colourings.
\\
\\
\noindent
We have seen that $G$ has exactly one component $G_1$ and that this component is $K_{\delta,|G_1|-\delta}$. In other words, if $G$ has the maximum number of $H$-colourings, then $G=K_{\delta,n-\delta}$ as required.
\end{proof}

\section{Conclusion}
\label{sec:conclusion}
We have shown that, given any graph $H$ and any $\delta\geq3$, for sufficiently large $n$, the graph $G=K_{\delta,n-\delta}$ maximises $\hom(G,H)$ among all connected graphs on $n$ vertices with minimum degree $\delta$. If $H$ has a component which is neither a complete looped graph nor a complete balanced bipartite graph, then $K_{\delta,n-\delta}$ is the unique such maximising graph.

We have also considered the more general question which was asked by Engbers \cite{E17}: what happens if we consider all graphs on $n$ vertices with minimum degree $\delta$, rather than just those which are connected? We will look at the case where $H$ is fixed and $\delta\geq\delta_0(H)$. By making $\delta$ sufficiently large in relation to $|H|$, we are able to identify the maximising graph for certain graphs $H$.

In what follows, we take $G$ to be any graph on $n$ vertices with minimum degree $\delta$. We assume that $G$ has $t$ components $G_1,\dots G_t$.

If $H$ is fixed with maximum degree $k$ and $\delta$ is sufficiently large, then the graph which maximises the number of $H$-colourings depends on the structure of $H$. Some of the different possible graphs which maximise $\hom(G,H)$ are given below.

\begin{enumerate}
\item \textit{$H$ is $h$ disjoint copies of the complete looped graph on $k$ vertices.}\\
It is easy to see that $\hom(G,H)=\prod_{i=1}^t|V(H)|k^{|G_i|-1}=h^tk^n$. When $h=1$, $\hom(G,H)=k^n$ for any graph $G$ on $n$ vertices and so every graph $G$ maximises the number of $H$-colourings. When $h>1$, $\hom(G,H)$ is maximised when $G$ has as many components as possible. The minimum number of vertices in a component of $G$ is $\delta+1$ which occurs when the component is $K_{\delta+1}$. Writing $n=a(\delta+1)+b$ where $b\in\{0,\dots\delta\}$, we have that $\hom(G,H)$ is maximised by any graph with $a$ components, e.g. $(a-1)K_{\delta+1}\cup K_{\delta+b+1}$.

\item \textit{$H$ is $h$ disjoint copies of $K_{k,k}$.}\\
It is easy to see that, if a graph is not bipartite, it is not possible to map it into $H$. Therefore
\[
\hom(G,H)=
\begin{cases}
\prod_{i=1}^t\hom(G_i,H)=(2h)^tk^n&\text{if $G_i$ is bipartite}\\
0&\text{if $G_i$ is not bipartite.}
\end{cases}
\]
Clearly, the number of $H$-colourings is maximised when $G$ is bipartite and has as many components as possible. The smallest possible bipartite component of $G$ is $K_{\delta,\delta}$ which has $2\delta$ vertices. Writing $n=2a\delta+b$ where $b\in\{0,\dots2\delta-1\}$, we have that $\hom(G,H)$ is maximised by any bipartite graph with $a$ components, e.g. $(a-1)K_{\delta,\delta}\cup K_{\delta,\delta+b}$.

\item \textit{No component of $H$ is the complete looped graph on $k$ vertices or $K_{k,k}$.}\\
In Section \ref{sec:largedelta}, we showed that, for any $\delta\geq\delta_0(H)$, there exists a constant $n_0(\delta,H)$ such that, if $n\geq n_0(\delta,H)$, then $K_{\delta,n-\delta}$ uniquely maximises the number of $H$-colourings.
\end{enumerate}
\noindent
From the examples given above, it is clear to see that there is not a simple answer to the question of which graph $G$ maximises $\hom(G,H)$ when $H$ is fixed and $\delta$ is sufficiently large. We make the following conjecture.

\begin{conjecture}
For any graph $H$ and any $\delta\geq\delta_0(H)$, there exists a constant $n_0(\delta,H)$ such that the following holds: if $G$ is a graph with minimum degree $\delta$ and at least $n_0(\delta,H)$ vertices, then
\[
\hom(G,H)\leq\max\big\{\hom(K_{\delta+1},H)^{\frac{|G|}{\delta+1}},\hom(K_{\delta,\delta},H)^{\frac{|G|}{2\delta}},\hom(K_{\delta,|G|-\delta},H)\big\}.
\]
\end{conjecture}
\noindent
This conjecture implies that, for a fixed graph $H$ and $\delta$ sufficiently large, the following holds: for sufficiently large $n$ satisfying suitable divisibility conditions, the number of $H$-colourings is always maximised by one of $\frac n{\delta+1}K_{\delta+1}$, $\frac n{2\delta}K_{\delta,\delta}$ or $K_{\delta,n-\delta}$.


\begin{thebibliography}{99}
\bibitem{C12} J. Cutler, Coloring graphs with graphs: a survey, {\em Graph Theory Notes NY} {\bf 63} (2012), 7--16.
\bibitem{CR14} J. Cutler and A.J. Radcliffe, The maximum number of complete subgraphs in a graph with given maximum degree, {\em J Combin Theory Ser B} {\bf 104} (2014), 60--71.
\bibitem{D10} R. Diestel, Graph Theory, Springer (2010).
\bibitem{E15} J. Engbers, Extremal $H$-colourings of graphs with fixed minimum degree, {\em J Graph Theory} {\bf 79} (2015), 103--124.
\bibitem{E17} J. Engbers, Maximizing $H$-colourings of connected graphs with fixed minimum degree, {\em J Graph Theory} {\bf 85} (2017), 780--787.
\bibitem{G13} D. Galvin, Maximising $H$-colourings of regular graphs, {\em J Graph Theory} {\bf 73} (2013), 66--84.
\bibitem{LPS10} P.-S. Loh, O. Pikhurko and B. Sudakov, Maximizing the number of $q$-colorings, {\em Proceedings of the London Mathematical Society} {\bf 101} (2010), 655--696.
\bibitem{S18} L. Sernau, Graph operations and upper bounds on graph homomorphism counts, {\em J Graph Theory} {\bf 87} (2018), 149--163.
\end{thebibliography}
\end{document}